\RequirePackage[resetfonts]{cmap}
\documentclass{amsart}
\pdfoutput=1

\usepackage{amsmath}
\usepackage[a4paper]{geometry}
\usepackage{amssymb}
\usepackage{amsthm}
\usepackage{enumerate}
\usepackage{verbatim}
\usepackage{perpage}
\usepackage{datetime}
\usepackage{nnfootnote}
\usepackage{color}
\usepackage{mathabx}
\usepackage{hyphenat}
\usepackage[pdftex, bookmarks, pdfpagelabels]{hyperref}
\newcommand{\papertitle}{Boundedness of Maximal Operators of Schr\"{o}dinger Type with Complex Time}
\definecolor{linkcolor}{rgb}{0.2,0,0.5}
\hypersetup{colorlinks=true,
						linkcolor=linkcolor,
						citecolor=linkcolor, 
						pdftitle={\papertitle}, 
						pdfauthor={Andrew David Bailey}, 
						pdfsubject={Fourier Analysis}, 
						pdfkeywords={Schr\"{o}dinger maximal operator, Schr\"odinger equation, KdV equation, complex time, boundedness, Sobolev space, pointwise convergence}, 
						pdfstartview=FitH,
						pdfdisplaydoctitle=true,
						pdfpagemode=UseNone,
						pdfborder=0 0 0}
\pdfpageattr {/Group << /S /Transparency /I true /CS /DeviceRGB>>}
\usepackage[T2A, OT1]{fontenc}
\usepackage{fancyhdr}

\theoremstyle{plain}
 \newtheorem{thm}{Theorem}[section] 
 \newtheorem*{thm*}{Theorem}
 
 \newtheorem{lemma}[thm]{Lemma}
 
 \newtheorem*{prop*}{Proposition}

\newdateformat{adbdate}{%
\dayofweekname{\THEDAY}{\THEMONTH}{\THEYEAR}  {\ordinal{DAY}} \monthname[\THEMONTH] \THEYEAR}
  \adbdate
  \settimeformat{xxivtime}
  
\renewcommand{\|}{\displaystyle}
\renewcommand{\(}{\left(}
\renewcommand{\)}{\right)}
\newcommand{\fhat}{\widehat{f}}

\renewcommand{\Gamma}{\varGamma}
\renewcommand{\epsilon}{\varepsilon}
\renewcommand{\bar}{\overline}

\renewcommand{\leq}{\leqslant}
\renewcommand{\geq}{\geqslant}
\MakePerPage{footnote} 

\DeclareMathOperator{\sgn}{sgn}

\newcommand{\supcite}[2][]{%
  \ifthenelse{\isempty{#1}}%
    {$^{\textrm{\cite{#2}}}$}%
    {$^{\textrm{\cite[#1]{#2}}}$}%
}
\newcommand{\notlabel}[1]{} 

\DeclareMathOperator{\loc}{loc}

\begin{document}
\title[\papertitle]{\papertitle}
\author[]{Andrew D. Bailey}
\address{Andrew D. Bailey, School of Mathematics, Watson Building, University of Birmingham, Edgbaston, Birmingham, B15 2TT, United Kingdom.}
\email{baileya@maths.bham.ac.uk}
\nnfoottext{Last updated on {\today} at \currenttime.\\
2010 Mathematics Subject Classification: 42B15, 42B25 (primary); 42B37 (secondary).\\
Keywords: Schr\"{o}dinger maximal operator, Schr\"odinger equation, KdV equation, complex time, boundedness, Sobolev space, pointwise convergence.\\
The author was supported by an EPSRC doctoral training grant.}
\begin{abstract}
Results of P. Sj\"olin and F. Soria on the Schr\"odinger maximal operator with complex-valued time are improved by determining up to the endpoint the sharp $s \geq 0$ for which boundedness from the Sobolev space $H^s(\mathbb{R})$ into $L^2(\mathbb{R})$ occurs.  Bounds are established for not only the Schr\"odinger maximal operator, but further for a general class of maximal operators corresponding to solution operators for certain dispersive PDEs.  As a consequence of additional bounds on these maximal operators from $H^s(\mathbb{R})$ into $L^2([-1, 1])$, sharp results on the pointwise almost everywhere convergence of the solutions of these PDEs to their initial data are determined.
\end{abstract}
\maketitle
\section{Introduction} \label{secintroprob}\thispagestyle{empty}
For $a > 1$, time parameter $t > 0$ and spatial variable $x \in \mathbb{R}$, define the following operator acting on functions $f$ in the Schwartz class, $\mathcal{S}(\mathbb{R})$:
\begin{equation*}
S^t_a f(x) \coloneq \int_{\mathbb{R}} \fhat(\xi) e^{i t |\xi|^a} e^{i x \xi} \, d\xi.
\end{equation*}
This operator gives the solution to the dispersive equation
\begin{equation*}
i \partial_t u(t, x) + (-\Delta_x)^{\frac{a}{2}} u(t, x) = 0
\end{equation*}
in one spatial dimension and one temporal dimension with initial data $u(0, x) = 2\pi f(x)$.  For $a = 2$, it corresponds to the solution operator for the Schr\"odinger equation.

The boundedness of the maximal operator $\| S^{*}_a f \coloneq \sup_{t \in (0, 1)} | S^t_a f |$ was considered by L. Carleson in the case of $a = 2$ in \cite{carleson}, motivated by the problem of determining what class of functions the operator $S^t_2$ can be defined on such that for all functions $f$ in this class, $\| \lim_{t \rightarrow 0} S^t_2 f(x) = f(x)$ for almost every $\| x \in \mathbb{R}$.  Whilst Carleson considered the problem in the context of H\"older continuous functions, an immediate consequence of his work was that when $f$ is taken in the Sobolev space $\| H^{\frac{1}{4}}(\mathbb{R})$, the following operator bound may be established:
\begin{equation*}
\Vert S^{*}_2 f \Vert_{L^2([-1, 1])} \lesssim \Vert f \Vert_{H^{\frac{1}{4}}(\mathbb{R})}.
\end{equation*}
Here and throughout this paper, the symbol $\lesssim$ is used to signify that the left hand side is bounded above by a constant multiple of the right hand side, with this constant independent of $f$.  In what follows, $A \approx B$ will be used to signify that $A$ is equal to a constant multiple of $B$, whilst $A \sim B$ will be used to signify that $A \lesssim B$ and $B \lesssim A$ both hold.

In \cite{dahlbergkenig}, B. Dahlberg and C. Kenig showed by exhibiting a counterexample that Carleson's result is sharp in the sense that for any positive $s < \frac{1}{4}$, there exists a function $\| f \in H^s$ such that $\| \lim_{t \rightarrow 0} |S^t_2 f(x)| = \infty$ for almost every $\| x \in \mathbb{R}$. As a consequence, $H^{\frac{1}{4}}$ in Carleson's maximal estimate can not be replaced with $H^s$ for any $s < \frac{1}{4}$.

The problem of boundedness of the Schr\"odinger maximal operator was considered in the context of global $\| L^2$ norms by L. Vega in \cite{vegatesis} where it was shown that
\begin{equation*}
\Vert S^{*}_2 f \Vert_{L^2(\mathbb{R})} \lesssim \Vert f \Vert_{H^s(\mathbb{R})}
\end{equation*}
for any $s > \frac{1}{2}$ and that this estimate fails for $s < \frac{1}{2}$.  The problem of boundedness for $s = \frac{1}{2}$ remains open.

It is a simple observation that if the definition of the solution operator for the Schr\"odinger equation is extended to allow complex-valued time with positive imaginary part, then for $t \geq 0$, the operator $S^{it}_2$ is the solution operator for the heat equation,
\begin{equation*}
\partial_t u(t, x) = \partial_x^2 u(t, x).
\end{equation*}
By observing that the multiplier associated with $S_2^{it}$ is a Gaussian, it follows from the boundedness of the Hardy--Littlewood maximal function that
\begin{equation*}
\Vert \sup_{t \in (0, 1)} |S^{it}_2 f| \Vert_{L^2(\mathbb{R})} \lesssim \Vert f \Vert_{L^2(\mathbb{R})}.
\end{equation*}

Given this result and the result of Vega, it is natural to consider intermediate results by asking for which $s > 0$ and which maps $g : [0, 1] \rightarrow [0, 1]$ with $\| \lim_{t \rightarrow 0} g(t) = 0$ it can be said that
\begin{equation*}
\Vert \sup_{t \in (0, 1)} |S^{t + ig(t)}_2 f | \Vert_{L^2(\mathbb{R})} \lesssim \Vert f \Vert_{H^s(\mathbb{R})}.
\end{equation*}

This interesting question was posed and partially answered by P. Sj\"{o}lin in \cite{sjolin}.  For $t$, $\gamma > 0$, he considered the operator
\begin{equation*}
P^t_{2, \gamma} f(x) \coloneq S^{t + it^{\gamma}}_2 f(x) = \int_{\mathbb{R}} \fhat(\xi) e^{i t \xi^2} e^{-t^{\gamma} \xi^2} e^{ix\xi} \, d\xi
\end{equation*}
with corresponding maximal operator $\| P^{*}_{2, \gamma} f \coloneq \sup_{t \in (0, 1)} |P^t_{2, \gamma} f|$.  Denoting by $s_2(\gamma)$ the infimum of the values of $s > 0$ such that the estimate $\Vert P^{*}_{2, \gamma} f \Vert_{L^2(\mathbb{R})} \lesssim \Vert f \Vert_{H^s(\mathbb{R})}$ holds, the following is established by Sj\"olin in \cite{sjolin} and in \cite{sjolinsoria} together with F. Soria:
\renewcommand{\labelenumi}{(\roman{enumi})}
\begin{thm} \label{sjolinsoriathm}\ 
\vspace{-0.3cm}
\begin{enumerate}
\item For $\gamma \in (0, 1]$, $s_2(\gamma) = 0$;
\item For $\gamma \in (1, 2)$, $s_2(\gamma) \in [0, \frac{1}{2} - \frac{1}{2\gamma}]$;
\item $s_2(2) = \frac{1}{4}$;
\item For $\gamma \in (2, 4]$, $s_2(\gamma) \in [\frac{1}{4}, \frac{1}{2} - \frac{1}{2\gamma}]$;
\item For $\gamma \in (4, \infty)$, $s_2(\gamma) \in [\frac{1}{2} - \frac{1}{\gamma}, \frac{1}{2} - \frac{1}{2\gamma}]$.
\end{enumerate}
\end{thm}
Sj\"olin's original work in \cite{sjolin} established cases (i) and (iii), as well as (v) with an upper bound of $\frac{1}{2}$ instead of $\frac{1}{2} - \frac{1}{2\gamma}$.  Cases (ii) and (iv) and the improvement of case (v) were established in \cite{sjolinsoria} with Soria.

In a paper from 1994, \cite{sjolinreal}, Sj\"olin also established that for any $a > 1$, the infimum of the values of $s > 0$ for which
\begin{equation*}
\Vert S^{*}_a f \Vert_{L^2(\mathbb{R})} \lesssim \Vert f \Vert_{H^s(\mathbb{R})}
\end{equation*}
holds is $\frac{a}{4}$.  As such, in addition to the problem of fully determining $s_2(\gamma)$ in cases (ii), (iv) and (v), it is natural to consider whether the scope of Theorem \ref{sjolinsoriathm} can be extended to $a \neq 2$.  To this end, for $t$, $\gamma > 0$ and $a > 1$, define in the natural way,
\begin{equation*}
P^t_{a, \gamma} f(x) \coloneq S^{t + it^{\gamma}}_a f(x) = \int_{\mathbb{R}} \fhat(\xi) e^{i t |\xi|^a} e^{-t^{\gamma} |\xi|^a} e^{ix\xi} \, d\xi
\end{equation*}
with corresponding maximal operator $\| P^{*}_{a, \gamma} f \coloneq \sup_{t \in (0, 1)} |P^t_{a, \gamma} f|$ and let $s_a(\gamma)$ denote the infimum of the values of $s > 0$ such that the estimate $\Vert P^{*}_{a, \gamma} f \Vert_{L^2(\mathbb{R})} \lesssim \Vert f \Vert_{H^s(\mathbb{R})}$ holds.  Letting $x^{+}$ denote the maximum of $0$ and $x$ for each $x \in \mathbb{R}$, in this paper, the following is established:
\begin{thm} \label{mainthm}
For $\gamma \in (0, \infty)$ and $a > 1$, $s_a(\gamma) = \frac{1}{4} a \(1 - \frac{1}{\gamma}\)^{+}$.
\end{thm}
In the case of $a = 2$, this ``completes'' Theorem \ref{sjolinsoriathm}.

By standard arguments, pointwise almost everywhere convergence of $P^t_{a, \gamma} f$ to $f$ as $t \rightarrow 0$ can be deduced as a corollary of Theorem \ref{mainthm} for functions $f \in H^s(\mathbb{R})$ when $a$ and $\gamma$ are as in the hypotheses of the theorem and $s > s_a(\gamma)$.  In fact, the following stronger result will also be established:
\begin{samepage}\begin{thm} \label{convcorl2}
For each $\gamma \in (0, \infty)$ and $a > 1$, the infimum of the values of $s > 0$ for which
\begin{equation*}
\lim_{t \rightarrow 0} P^t_{a, \gamma} f(x) = f(x)
\end{equation*}
for almost every $x \in \mathbb{R}$ whenever $f \in H^s(\mathbb{R})$ is $\min\(\frac{1}{4} a \(1 - \frac{1}{\gamma}\)^{+}, \frac{1}{4}\)$.  Moreover, this convergence also occurs for all $f \in L^2(\mathbb{R})$ when $\gamma \in (0, 1]$ and for all $f \in H^{\frac{1}{4}}(\mathbb{R})$ when $\gamma \in [\frac{a}{a-1}, \infty)$.
\end{thm}\end{samepage}
Theorem \ref{convcorl2} will be proved in Section \ref{furtherremarks} as a consequence of local bounds for $P^{*}_{a, \gamma}$ deduced from Theorem \ref{mainthm} and some other previous work of Sj\"olin.  To prove Theorem \ref{mainthm}, it will suffice to consider the case of $\gamma > 1$, owing to the following generalisation of Sj\"olin's Lemma 1 from \cite{sjolin}:
\begin{lemma} \label{boundinglemma}
Let $g$ and $h$ be continuous functions mapping $[0, 1]$ to $[0, 1]$ such that $g(t) \leq h(t)$ for all $t \in (0, 1)$.  Then for any $a > 1$,
\begin{equation*}
\Vert \sup_{t \in (0, 1)} |S_a^{t + ih(t)}f| \Vert_{L^2(\mathbb{R})} \lesssim \Vert \sup_{t \in (0, 1)} |S_a^{t + ig(t)} f| \Vert_{L^2(\mathbb{R})}
\end{equation*}
for any $f \in \mathcal{S}(\mathbb{R})$.
\end{lemma}
In addition to reducing the proof of Theorem \ref{mainthm} to the case of $\gamma > 1$, this lemma also suggests that in terms of understanding the convergence at the origin, the $P^{*}_{a, \gamma}$ are natural operators to consider as they encapsulate the convergence properties of any operator of the form $S_a^{t + ih(t)}f$ when $h(t)$ is of polynomial type near $t = 0$.  The proof is essentially the same as the proof of the analogous result from \cite{sjolin} and will thus be given in an appendix.

The proof of Theorem \ref{mainthm} will be divided into two sections.  In Section \ref{possec}, it will be shown that $\Vert P^{*}_{a, \gamma} f \Vert_{L^2(\mathbb{R})} \lesssim \Vert f \Vert_{H^s(\mathbb{R})}$ holds for all $s$ above the critical index, $\frac{1}{4} a \(1 - \frac{1}{\gamma}\)$, when $\gamma > 1$, whilst in Section \ref{negsec}, it will be shown that this boundedness fails for all $s$ below this index.  Section \ref{furtherremarks} will contain some further remarks on the implications of Theorem \ref{mainthm} and its proof, as well as a proof of Theorem \ref{convcorl2}.

The work contained in this paper will form part of the author's forthcoming doctoral thesis, \cite{phd}.  The author would like to express his gratitude to his supervisor, Jonathan Bennett, for all his support and assistance over the last few years.  The author is also grateful to Neal Bez and Keith Rogers for some valuable discussions about this work.

\section{Proof of Boundedness of the Maximal Operator for Regularity Above the Critical Index} \label{possec}
It is claimed that to show that $\Vert P^{*}_{a, \gamma} f \Vert_{L^2(\mathbb{R})} \lesssim \Vert f \Vert_{H^s(\mathbb{R})}$ for all $s > \frac{1}{4} a \(1 - \frac{1}{\gamma}\)$, it will be sufficient to prove the following lemma:
\begin{lemma} \label{mainlemma}
Suppose that $a$, $\gamma > 1$ and $\alpha > \frac{1}{2} a\(1 - \frac{1}{\gamma}\)$.  If $\gamma < \frac{a}{a-1}$, suppose further that $\alpha < \frac{1}{2}$.  Let $\| \mu \in \mathcal{S}(\mathbb{R})$ be compactly supported, positive, even and real-valued.  Then there exists $K \in L^1(\mathbb{R})$ such that for any $t_1$, $t_2 \in (0, 1)$ and $N \in \mathbb{N}$,
\begin{equation*}
\bigg|\int_{\mathbb{R}} e^{i((t_1 - t_2)|\xi|^a - x\xi)} (1 + \xi^2)^{-\frac{\alpha}{2}} e^{-(t_1^{\gamma} + t_2^{\gamma})|\xi|^a} \mu\(\frac{\xi}{N}\) \, d\xi \bigg| \leq K(x)
\end{equation*}
for all $x \in \mathbb{R}$.
\end{lemma}
It is remarked that the assumption that $\alpha < \frac{1}{2}$ for $\gamma < \frac{a}{a-1}$ is purely for technical reasons and since only minimal choices of $\alpha$ are of interest in proving Theorem \ref{mainthm}, it will have no impact on the usefulness of this lemma.

The sufficiency of this lemma in establishing the desired boundedness can be shown using the Kolmogorov--Seliverstov--Plessner method.  Indeed, assuming this lemma to be true, fix any positive, even $\eta \in \mathcal{S}(\mathbb{R})$ supported in $[-1, 1]$ and equal to $1$ in $[-\frac{1}{2}, \frac{1}{2}]$.  Also, fix a measurable function $t: \mathbb{R} \rightarrow (0, 1)$ and define for each $N \in \mathbb{N}$,
\begin{equation*}
P_{a, \gamma, N}^{t(x)} f(x) \coloneq \eta\(\frac{x}{N}\) \int_{\mathbb{R}} \fhat(\xi) e^{it(x)|\xi|^a} e^{-t^{\gamma}(x)|\xi|^a} e^{ix\xi} \eta\Big(\frac{\xi}{N}\Big)\, d\xi.
\end{equation*} 
To establish that $\Vert P^{*}_{a, \gamma} f \Vert_{L^2(\mathbb{R})} \lesssim \Vert f \Vert_{H^s(\mathbb{R})}$ for all $s > \frac{1}{4} a \(1 - \frac{1}{\gamma}\)$, it suffices to prove that
\begin{equation*}
\Vert P^{t(\cdot)}_{a, \gamma, N} f \Vert_{L^2(\mathbb{R})} \lesssim \Vert f \Vert_{H^s(\mathbb{R})}
\end{equation*}
for any $N \in \mathbb{N}$ with constant independent of $N$ and $t$.  This is equivalent to showing that for any $g \in L^2(\mathbb{R})$ with $\Vert g \Vert_{L^2(\mathbb{R})} = 1$,
\begin{equation*}
\Big| \int_{\mathbb{R}} P^{t(x)}_{a, \gamma, N} f(x) \bar{g(x)} \, dx \Big| \lesssim \Vert f \Vert_{H^s(\mathbb{R})}.
\end{equation*}
However, by Fubini's Theorem and the Cauchy--Schwarz inequality,
\begin{eqnarray*}
&& \Big| \int_{\mathbb{R}} P^{t(x)}_{a, \gamma, N}f(x) \bar{g(x)} \, dx \Big|\\*
&=& \Big| \int_{\mathbb{R}} \fhat(\xi) (1 + \xi^2)^{\frac{s}{2}} (1 + \xi^2)^{-\frac{s}{2}} \eta\Big(\frac{\xi}{N}\Big) \int_{\mathbb{R}} e^{it(x)|\xi|^a} e^{-t^{\gamma}(x) |\xi|^a} e^{ix \xi} \bar{g(x)} \eta\(\frac{x}{N}\) \, dx \, d\xi \Big|\\*
&\leq& \Vert f \Vert_{H^s(\mathbb{R})} \Big| \int_{\mathbb{R}} (1 + \xi^2)^{-s} \eta^2\Big(\frac{\xi}{N}\Big) \int_{\mathbb{R}} \int_{\mathbb{R}} e^{i(t(x) - t(y))|\xi|^a} e^{-(t^{\gamma}(x)+t^{\gamma}(y))|\xi|^a} e^{i(x - y)\xi} \bar{g(x)} g(y)\\*
&& \quad {} \times \eta\(\frac{x}{N}\) \eta\(\frac{y}{N}\) \, dx \, dy \, d\xi \Big|^{\frac{1}{2}}\\
&\lesssim& \Vert f \Vert_{H^s(\mathbb{R})} \Big( \int_{\mathbb{R}} \int_{\mathbb{R}} |g(x)| |g(y)| \Big| \int_{\mathbb{R}} e^{i((t(x) - t(y))|\xi|^a - (y-x)\xi)} (1 + \xi^2)^{-s} e^{-(t^{\gamma}(x) + t^{\gamma}(y))|\xi|^a}\\*
&& \quad {} \times \eta^2\Big(\frac{\xi}{N}\Big) \, d\xi \Big| \, dx \, dy \Big)^{\frac{1}{2}}.
\end{eqnarray*}
By Lemma \ref{mainlemma} (where $\alpha = 2s$, $\mu = \eta^2$, $t_1 = t(x)$ and $t_2 = t(y)$) and a further application of the Cauchy--Schwarz inequality, this quantity can be bounded by $\| \Vert f \Vert_{H^s(\mathbb{R})} \Vert |K| * |g| \Vert_{L^2(\mathbb{R})}^{\frac{1}{2}} \Vert g \Vert_{L^2(\mathbb{R})}^{\frac{1}{2}}$, which, by Young's convolution inequality and the fact that $\Vert g \Vert_{L^2(\mathbb{R})} = 1$, is bounded by $\Vert f \Vert_{H^s(\mathbb{R})} \Vert K \Vert_{L^1(\mathbb{R})}^{\frac{1}{2}}$.  This establishes the desired boundedness of $P^{*}_{a, \gamma}$.

Lemma \ref{mainlemma} is based on Lemmata 2.1, 2.2, 2.3 and 2.4 from \cite{sjolinsoria} and its proof given here follows a similar strategy to the proofs of those lemmata.  Note that of these four lemmata, Lemma 2.1 was proved in \cite{sjolin2007} (where it is cited as having been originally proved implicitly in \cite{gulkan} using a method from \cite{sjolinarticle}) and Lemma 2.2 was proved in \cite{sjolin}.

The remainder of this section will be devoted to the proof of Lemma \ref{mainlemma}.

To begin with, for each $\epsilon > 0$, define the function $h_{\epsilon}(\xi) \coloneq e^{-\epsilon |\xi|^a}$.  It is claimed that for $\xi \neq 0$,
\begin{equation*}
|h_{\epsilon}'(\xi)| \lesssim \frac{1}{|\xi|} \quad \mbox{and} \quad |h_{\epsilon}''(\xi)| \lesssim \frac{1}{|\xi|^2}
\end{equation*}
with constant independent of $\epsilon$.  Indeed, note that $h_{\epsilon}'(\xi) = -\sgn(\xi) \epsilon a |\xi|^{a-1} e^{-\epsilon |\xi|^a}$, so \begin{equation*}
|h_{\epsilon}'(\xi)| \leq \frac{a}{|\xi|} ( \max_{y \in \mathbb{R}^{+}} y e^{-y} ) \lesssim \frac{1}{|\xi|}.
\end{equation*}
Similarly,
\begin{equation*}
|h_{\epsilon}''(\xi)| \lesssim \frac{1}{|\xi|^2} (\max_{y \in \mathbb{R}^{+}} y e^{-y} ) + \frac{1}{|\xi|^2} (\max_{y \in \mathbb{R}^{+}} y^2 e^{-y}) \lesssim \frac{1}{|\xi|^2}.
\end{equation*}

Now, assume without loss of generality that $t_2 < t_1$ and set $t \coloneq t_1 - t_2$ and $\epsilon \coloneq t_1^{\gamma} + t_2^{\gamma}$.  Also, define $F(\xi) \coloneq t|\xi|^a - x\xi$ and $G(\xi) \coloneq (1 + \xi^2)^{-\frac{\alpha}{2}} e^{-\epsilon|\xi|^a} \mu\(\frac{\xi}{N}\)$, so that the integral in Lemma \ref{mainlemma} can be rewritten as
\begin{equation*}
\int_{\mathbb{R}} e^{iF(\xi)} G(\xi) \, d\xi.
\end{equation*}
The letter $\rho$ will be used to denote $\(\frac{|x|}{ta}\)^{\frac{1}{a-1}}$, a (possibly) stationary point of $F$.

Fixing a large constant $C_0 \in \mathbb{R}^{+}$, the cases of $|x| \leq C_0$ and $|x| \geq C_0$ will be considered separately.

\subsection{\texorpdfstring{The Case of $|x| \leq C_0$}{The Case of Small x}}
Split the integral as $A + B$ where
\begin{eqnarray*}
A &\coloneq& \int_{|\xi| \leq |x|^{-1}} e^{i(t|\xi|^a - x \xi)} (1+\xi^2)^{-\frac{\alpha}{2}} e^{-\epsilon |\xi|^a} \mu\(\frac{\xi}{N}\) \, d\xi,\\*
B &\coloneq& \int_{|\xi| \geq |x|^{-1}} e^{i(t|\xi|^a - x \xi)} (1+\xi^2)^{-\frac{\alpha}{2}} e^{-\epsilon |\xi|^a} \mu\(\frac{\xi}{N}\) \, d\xi.
\end{eqnarray*}

The first integral, $A$, can be bounded trivially by simply observing that
\begin{equation*}
|A| \lesssim \int_{|\xi| \leq |x|^{-1}} (1+ \xi^2)^{-\frac{\alpha}{2}} \, d\xi \lesssim 1 + |x|^{\alpha - 1} \lesssim |x|^{\min(0, \alpha - 1)}.
\end{equation*}

Since $\min(0, \alpha - 1) > -1$, the required estimate on $A$ is established.

To estimate $B$, first assume that $|x|^a \leq \frac{t}{2}$.  This ensures that the phase of the integrand (the function $F$) is never stationary in the region of integration for $B$.

By symmetry, it will suffice to bound $B$ with the range of integration restricted to positive values of $\xi$.  By direct calculation, for such $\xi$, $F'(\xi) = a t \xi^{a-1} - x$, so it can be seen that $F'$ is monotonic.  Further, given that $|\xi| \geq |x|^{-1}$ and $\frac{t}{|x|^{a-1}} \geq 2|x|$,
\begin{equation*}
|F'(\xi)| \geq |x| (2a - 1) > |x|,
\end{equation*}
so by Van der Corput's Lemma, it follows that
\begin{equation*}
\bigg|\int_{\xi > |x|^{-1}} e^{iF(\xi)} G(\xi) \, d\xi \bigg| \lesssim \frac{1}{|x|} \(\sup_{\xi > |x|^{-1}} |G(\xi)| + \int_{\xi > |x|^{-1}} |G'(\xi)| \, d\xi \).
\end{equation*}

Trivially, $|G(\xi)| \lesssim (1 + \xi^2)^{-\frac{\alpha}{2}} \lesssim |x|^{\alpha}$ for $\xi > |x|^{-1}$.  Recalling that $h_\epsilon(\xi) \coloneq e^{-\epsilon|\xi|^a}$,
\begin{equation*}
G'(\xi) = 2\xi\(-\frac{\alpha}{2}\)(1+\xi^2)^{-\frac{\alpha}{2} - 1} h_\epsilon(\xi) \mu\(\frac{\xi}{N}\) + (1+ \xi^2)^{-\frac{\alpha}{2}} h_{\epsilon}'(\xi) \mu\(\frac{\xi}{N}\) + (1 + \xi^2)^{-\frac{\alpha}{2}} h_{\epsilon}(\xi) \frac{1}{N} \mu'\(\frac{\xi}{N}\),
\end{equation*}
so since $|h_{\epsilon}'(\xi)| \lesssim \frac{1}{|\xi|}$ with constant independent of $\epsilon$,
\begin{eqnarray*}
|G'(\xi)| &\leq& \alpha \xi (1 + \xi^2)^{-\frac{\alpha}{2} - 1} h_{\epsilon}(\xi) \mu\(\frac{\xi}{N}\) + (1 + \xi^2)^{-\frac{\alpha}{2}} |h_{\epsilon}'(\xi)| \mu\(\frac{\xi}{N}\) + (1 + \xi^2)^{-\frac{\alpha}{2}} h_{\epsilon}(\xi) \frac{1}{N} \Big|\mu'\(\frac{\xi}{N}\)\Big|\\*
&\lesssim& \xi^{-\alpha - 1} + \frac{\xi^{-\alpha}}{N} \Big|\mu'\(\frac{\xi}{N}\)\Big|\\
&\lesssim& \xi^{-\alpha - 1}.
\end{eqnarray*}
It follows that
\begin{equation*}
\int_{\xi > |x|^{-1}} |G'(\xi)| \, d\xi \leq \int_{|x|^{-1}}^{\infty} \xi^{-\alpha - 1} \, d\xi \lesssim |x|^{\alpha},
\end{equation*}
and hence $B \lesssim |x|^{\alpha - 1}$.  Again, given that $\alpha - 1 > -1$, the desired estimate for $B$ holds in the case of $|x|^a \leq \frac{t}{2}$.

To complete the proof of the lemma in the case of $|x| \leq C_0$, it remains to bound $B$ in the case that $|x|^a \geq \frac{t}{2}$.  As before, it suffices to consider only positive $\xi$.  To proceed, fix a small constant, $\delta$, and a large constant, $K$.  The range of integration will be split into the following regions:
\begin{eqnarray*}
I_1 &\coloneq& \{ \xi \geq |x|^{-1} : \xi \leq \delta \rho \},\\*
I_2 &\coloneq& \{ \xi \geq |x|^{-1} : \xi \in [\delta \rho, K\rho] \},\\*
I_3 &\coloneq& \{ \xi \geq |x|^{-1} : \xi \geq K \rho \},
\end{eqnarray*}
recalling that $\rho = \(\frac{|x|}{ta}\)^{\frac{1}{a-1}}$.  For each $j \in \{1, 2, 3\}$, the integral in $B$ restricted to the region $I_j$ will be denoted by $J_j$.

This splitting isolates a neighbourhood around the point of (possible) stationary phase of the integrand ($I_2$) from the remaining range of integration either side ($I_1$ and $I_3$).  The latter regions will be bounded using a lower bound on $F'$ and an application of Van der Corput's Lemma, as before.  Indeed, for $\xi \in I_1$, it can be seen that $at\xi^{a-1} \leq \delta^{a-1}|x| \leq \frac{|x|}{2}$, hence $|F'(\xi)| = |at\xi^{a-1} - x| \geq \frac{|x|}{2}$.  Similarly, for $\xi \in I_3$, it can be seen that $at\xi^{a-1} \geq K^{a-1}|x| \geq 2|x|$, so $|F'(\xi)| \geq \frac{|x|}{2}$.  From before,
\begin{equation*}
\sup_{\xi > |x|^{-1}} |G(\xi)| + \int_{\xi > |x|^{-1}} |G'(\xi)| \, d\xi \lesssim |x|^{\alpha},
\end{equation*}
so by Van der Corput's Lemma,
\begin{equation*}
|J_1|,\ |J_3| \lesssim |x|^{-1} |x|^{\alpha} = |x|^{\alpha - 1}.
\end{equation*}

Unsurprisingly, the estimate on $J_2$ is more delicate, although it is still attained using Van der Corput's Lemma, this time with the second derivative of $F$ bounded below.  To begin with, assume that $\gamma \geq \frac{a}{a-1}$.  For any $\xi$ in $I_2$, it is the case that $\xi \sim \rho$.  Given that $F''(\xi) = a(a-1)t \xi^{a-2}$, it follows that $|F''(\xi)| \gtrsim t^{\frac{1}{a-1}} |x|^{\frac{a-2}{a-1}}$.  Following the same method as before, but now using that $\xi \gtrsim \rho$ instead of simply that $\xi \geq |x|^{-1}$, it is also the case that
\begin{equation*}
\sup_{\xi \in I_2} |G(\xi)| \lesssim \rho^{-\alpha} \quad \mbox{and} \quad \int_{I_2} |G'(\xi)| \, d\xi \lesssim \rho^{-\alpha}.
\end{equation*}
Consequently, by Van der Corput's Lemma,
\begin{equation*}
|J_2| \lesssim t^{-\frac{1}{2(a-1)}} |x|^{-\frac{(a-2)}{2(a-1)}} \rho^{-\alpha} \approx t^{\frac{1}{a-1}(\alpha - \frac{1}{2})} |x|^{\frac{1}{a-1}(1 - \frac{1}{2}a - \alpha)}.
\end{equation*}
Since $\gamma \geq \frac{a}{a-1}$, it is necessarily the case that $\alpha - \frac{1}{2} > 0$.  Using further the assumption that $|x|^a \geq \frac{t}{2}$, it follows that
\begin{equation*}
t^{\frac{1}{a-1}(\alpha - \frac{1}{2})} \lesssim |x|^{\frac{a}{a-1}(\alpha - \frac{1}{2})},
\end{equation*}
so
\begin{equation*}
|J_2| \lesssim |x|^{\frac{a}{a-1}(\alpha - \frac{1}{2})} |x|^{\frac{1}{a-1} (1 - \frac{1}{2}a - \alpha)} = |x|^{\alpha - 1},
\end{equation*}
which completes the desired estimate.

It remains only to consider the case of $\gamma < \frac{a}{a-1}$.  The lower bound, $|F''(\xi)| \geq t^{\frac{1}{a-1}} |x|^{\frac{a-2}{a-1}}$ will be used again in another application of Van der Corput's Lemma, but instead of using the fact that $|x|^a \geq \frac{t}{2}$, improved estimates on $G$ will be required.  Indeed, note first that
\begin{equation*}
\sup_{\xi \in I_2} |G(\xi)| \lesssim \rho^{-\alpha} e^{-\delta^a \epsilon \rho^a}.
\end{equation*}
Also, similarly to before,
\begin{equation*}
|G'(\xi)| \lesssim \rho^{-\alpha} |h_{\epsilon}'(\xi)| + \rho^{-\alpha - 1}h_{\epsilon}(\delta \rho),
\end{equation*}
so
\begin{eqnarray*}
\int_{I_2} |G'(\xi)| \, d\xi &\lesssim& \rho^{-\alpha} \int_{\delta\rho}^{K\rho} |h_{\epsilon}'(\xi)| \, d\xi + \int_{\delta\rho}^{K\rho} \rho^{-\alpha - 1} h_{\epsilon} (\delta\rho) \, d\xi\\*
&=& -\rho^{-\alpha} \int_{\delta\rho}^{K\rho} h_{\epsilon}'(\xi) \, d\xi + \int_{\delta\rho}^{K\rho} \rho^{-\alpha - 1} h_{\epsilon}(\delta\rho) \, d\xi\\
&\approx& \rho^{-\alpha} e^{-\delta^a \epsilon \rho^a}
\end{eqnarray*}
by the Fundamental Theorem of Calculus.

As such, by Van der Corput's Lemma,
\begin{eqnarray*}
|J_2| &\lesssim& t^{-\frac{1}{2(a-1)}} |x|^{-\frac{a-2}{2(a-1)}} \rho^{-\alpha} e^{-\delta^a \epsilon \rho^a}\\*
&\approx& t^{\frac{1}{a-1}(\alpha - \frac{1}{2})} |x|^{\frac{1}{a-1}(-\alpha - \frac{1}{2}(a-2))} e^{-\delta^a(t_1^{\gamma} + t_2^{\gamma})|x|^{\frac{a}{a-1}} t^{-\frac{a}{a-1}}}.
\end{eqnarray*}
Further, observing that $t_1^{\gamma} + t_2^{\gamma} \gtrsim (t_1 + t_2)^{\gamma} \geq t^{\gamma}$, it can be seen that there exists a small constant $c_0 > 0$ such that
\begin{equation*}
|J_2| \lesssim t^{\frac{1}{a-1}(\alpha - \frac{1}{2})} |x|^{\frac{1}{a-1}(-\alpha - \frac{1}{2}(a-2))} e^{-\delta^a c_0 t^{\gamma - \frac{a}{a-1}}|x|^{\frac{a}{a-1}}}.
\end{equation*}
Noting that for any $y$, $\beta > 0$, the inequality
\begin{equation*}
e^{-y} \lesssim_{\beta} y^{-\beta}
\end{equation*}
holds, it follows that for any $\beta > 0$,
\begin{eqnarray*}
|J_2| &\lesssim& t^{\frac{1}{a-1} (\alpha - \frac{1}{2})} |x|^{\frac{1}{a-1} (-\alpha - \frac{1}{2}(a-2))} t^{-\beta(\gamma - \frac{a}{a-1})} |x|^{-\frac{\beta a}{a-1}}\\*
&=& \frac{t^{\frac{1}{a-1}(\alpha - \frac{1}{2})}}{t^{\beta(\gamma - \frac{a}{a-1})}} \frac{1}{|x|^{\frac{1}{a-1}(\alpha + \frac{1}{2}(a-2) + \beta a)}}.
\end{eqnarray*}
Choose $\beta$ such that $\frac{1}{a-1}(\alpha - \frac{1}{2}) = \beta(\gamma - \frac{a}{a-1})$, that is $\beta = \frac{\alpha - \frac{1}{2}}{(a-1)\gamma - a}$, noting that $\beta$ is genuinely positive as $\gamma < \frac{a}{a-1}$ and $\alpha < \frac{1}{2}$.  It follows that $|J_2| \lesssim \frac{1}{|x|^k}$ where $k = \frac{1}{a-1} \(\alpha + \frac{1}{2}(a-2) + \frac{a(\alpha - \frac{1}{2})}{(a-1)\gamma - a}\)$, so it remains only to show that $k < 1$.  However,
\begin{equation*}
k = \frac{1}{a-1}\(\alpha \(\frac{(a-1)\gamma}{(a-1)\gamma - a}\) + \frac{1}{2} (a-2) - \frac{\frac{1}{2}a}{(a-1)\gamma - a}\),
\end{equation*}
but $\gamma < \frac{a}{a-1}$, so $\frac{(a-1)\gamma}{(a-1)\gamma - a} < 0$, hence the fact that $\alpha > \frac{1}{2} a (1 - \frac{1}{\gamma})$ implies that
\begin{equation*}
k < \frac{1}{a-1} \(\frac{1}{2}a \(1 - \frac{1}{\gamma}\)\(\frac{(a-1)\gamma}{(a-1)\gamma - a}\) + \frac{1}{2}(a-2) - \frac{\frac{1}{2}a}{(a-1)\gamma - a}\) = 1.
\end{equation*}
The estimate for $|x| < C_0$ is thus established.

\subsection{\texorpdfstring{The Case of $|x| \geq C_0$}{The Case of Large x}}
Here again the integral will be split into four regions, this time smoothly partitioned.  To this end, define $\phi_0 \in \mathcal{S}(\mathbb{R})$ to be supported in $[-1, 1]$ and equal to $1$ in $[-\frac{1}{2}, \frac{1}{2}]$ and $\phi_2 \in \mathcal{S}(\mathbb{R})$ to be supported in $[\delta\rho, K\rho]$ and equal to $1$ in $[2\delta\rho, \frac{1}{2}K\rho]$, where, as before, $\delta$ is a small constant, $K$ is a large constant and $\rho = \(\frac{|x|}{ta}\)^{\frac{1}{a-1}}$.  For the sake of simplicity, it is assumed that $C_0$ and $\delta$ have been chosen so that $\delta \(\frac{|x|}{a}\)^{\frac{1}{a-1}} > 1$ (and hence $\delta \rho > 1$).  Define $\phi_3 \coloneq (1 - \phi_2) \chi_{[\frac{1}{2} K \rho, \infty)}$ and $\phi_1 \coloneq (1 - \phi_2 - \phi_0) \chi_{[\frac{1}{2}, 2\delta\rho]}$.  Further, define $G_j \coloneq G \phi_j$ and let $I_j$ represent the support of $G_j$ for each $j \in \{0, 1, 2, 3\}$, so that
\begin{eqnarray*}
I_0 &=& [-1, 1],\\*
I_1 &=& [\tfrac{1}{2}, 2\delta\rho],\\*
I_2 &=& [\delta\rho, K\rho],\\*
I_3 &=& [\tfrac{1}{2}K\rho, \infty).
\end{eqnarray*}
As before, this splitting isolates a region, $I_2$, around a point of possible stationary phase of the integrand from regions either side, $I_1$ and $I_3$.  The region $I_0$ has a similar role to integral $A$ from the previous section.

By symmetry it suffices to estimate
\begin{equation*}
J_j \coloneq \int e^{iF} G_j
\end{equation*}
for each $j \in \{0, 1, 2, 3\}$.

In the case of $J_0$, writing $e^{iF}G$ as $(e^{-ix\xi})(e^{it|\xi|^a}G_0(\xi))$ and integrating by parts twice yields that
\begin{equation*}
|J_0| \leq \frac{1}{x^2} \int_{-1}^1 \Big| \frac{d^2}{d\xi^2} (e^{it|\xi|^a} G_0(\xi))\Big| \, d\xi.
\end{equation*}
By direct calculation and the triangle inequality,
\begin{equation*}
\Big|\frac{d^2}{d\xi^2} (e^{it|\xi|^a} G_0(\xi))\Big| \lesssim |\xi|^{a-2} |G_0(\xi)| + |\xi|^{2a - 2}|G_0(\xi)| + |\xi|^{a-1}|G_0'(\xi)| + |G_0''(\xi)|.
\end{equation*}
Now, $G_0(\xi) = (1+\xi^2)^{-\frac{\alpha}{2}} e^{-\epsilon|\xi|^a}\mu\(\frac{\xi}{N}\) \phi_0(\xi)$, so for $\xi \in [-1, 1]$, it is clear that $|G_0(\xi)| \lesssim 1$.  Further, the first derivatives of all terms in the product defining $G_0$ are bounded for $\xi \in [-1, 1]$, so $|G_0'(\xi)| \lesssim 1$ also.  Finally, the second derivatives of all terms in the product defining $G_0$ are also bounded for $\xi \in [-1, 1]$ with the exception of $\frac{d^2}{d\xi^2} e^{-\epsilon |\xi|^a}$.  However, by the triangle inequality,
\begin{equation*}
\Big|\frac{d^2}{d\xi^2} e^{-\epsilon|\xi|^a}\Big| \leq \epsilon a(a-1) |\xi|^{a-2} e^{-\epsilon |\xi|^a} + (a \epsilon |\xi|^{a-1})^2 e^{-\epsilon |\xi|^a} \lesssim |\xi|^{a-2} + 1.
\end{equation*}
Given that $a > 1$, this expression is integrable on $[-1, 1]$, and it hence follows that
\begin{equation*}
\int_{-1}^1 \Big|\frac{d^2}{d\xi^2} (e^{it|\xi|^a} G_0(\xi))\Big| \, d\xi \lesssim 1,
\end{equation*}
so $|J_0| \lesssim x^{-2}$, completing the required estimate on $J_0$.

For $j \in \{1, 3\}$, integrating by parts twice yields that
\begin{eqnarray*}
|J_j| &=& \Big| \int_{I_j} e^{iF(\xi)} \( -\frac{{G_j}''(\xi)}{(F'(\xi))^2} + \frac{3{G_j}'(\xi)F''(\xi)}{(F'(\xi))^3} + \frac{G_j(\xi)F'''(\xi)}{(F'(\xi))^3} -\frac{3G_j(\xi)(F''(\xi))^2}{(F'(\xi))^4} \) \, d\xi \Big|\\*
&\lesssim& \int_{I_j} \frac{1}{(F'(\xi))^2} \( |{G_j}''(\xi)| + \frac{|F''(\xi)|}{|F'(\xi)|} |{G_j}'(\xi)| + \frac{|F'''(\xi)|}{|F'(\xi)|} |G_j(\xi)| + \frac{|F''(\xi)|^2}{|F'(\xi)|^2} |G_j(\xi)| \) \, d\xi.
\end{eqnarray*}

Given that $\xi > 0$, by direct calculation,
\begin{equation*}
F(\xi) = t\xi^a - x\xi, \quad F'(\xi) = at\xi^{a-1} - x, \quad F''(\xi) = a(a-1)t\xi^{a-2}, \quad F'''(\xi) = a(a-1)(a-2)t\xi^{a-3}.
\end{equation*}
For $\xi \in I_1$, $at\xi^{a-1} \leq at 2^{a-1} \delta^{a-1} \rho^{a-1} = 2^{a-1} \delta^{a-1} |x|$.  It follows that $|F'(\xi)| \gtrsim |x|$ and hence also that $|F'(\xi)| \gtrsim at\xi^{a-1}$.  Similarly, for $\xi \in I_3$, $at\xi^{a-1} \geq at 2^{1-a} K^{a-1} \rho^{a-1} = 2^{1-a} K^{a-1} |x|$, hence $|F'(\xi)| \gtrsim |x|$ and $|F'(\xi)| \gtrsim at\xi^{a-1}$ as well.  It follows in both cases that
\begin{equation*}
\frac{|F''(\xi)|}{|F'(\xi)|} \lesssim \xi^{-1} \quad \mbox{and} \quad \frac{|F'''(\xi)|}{|F'(\xi)|} \lesssim \xi^{-2}.
\end{equation*}
Since $h_{\epsilon}(\xi) \coloneq e^{-\epsilon |\xi|^a}$ satisfies the estimates $|h_{\epsilon}'(\xi)| \lesssim \frac{1}{|\xi|}$ and $|h_{\epsilon}''(\xi)| \lesssim \frac{1}{|\xi|^2}$, with constants independent of $\epsilon$, it is easily seen that for any $j \in \{1, 2, 3\}$,
\begin{equation*}
|G_j(\xi)| \lesssim \frac{1}{|\xi|^{\alpha}}, \quad |{G_j}'(\xi)| \lesssim \frac{1}{|\xi|^{\alpha + 1}}, \quad |{G_j}''(\xi)| \lesssim \frac{1}{|\xi|^{\alpha + 2}}.
\end{equation*}
Consequently, for $j \in \{1, 3\}$,
\begin{equation*}
|J_j| \lesssim \frac{1}{|x|^2} \int_{I_j} \frac{1}{|\xi|^{\alpha + 2}}\, d\xi \lesssim \frac{1}{|x|^2},
\end{equation*}
which completes the required estimates on $J_1$ and $J_3$.

To bound $J_2$, following the same steps as in the bound for $J_2$ when $|x| \leq C_0$, it can be seen that for any $\xi \in I_2$,
\begin{equation*}
|F''(\xi)| \gtrsim t^{\frac{1}{a-1}}|x|^{\frac{a-2}{a-1}}
\end{equation*}
and that
\begin{equation*}
\sup_{\xi \in I_2} |G_2(\xi)| + \int_{I_2} |{G_2}'(\xi)| \, d\xi \lesssim \rho^{-\alpha} e^{-\delta^a \epsilon \rho^a},
\end{equation*}
so by Van der Corput's Lemma,
\begin{equation*}
|J_2| \lesssim t^{\frac{1}{a-1} (\alpha - \frac{1}{2})} |x|^{\frac{1}{a-1}(-\alpha - \frac{1}{2}(a-2))} e^{-\delta^a c_0 t^{\gamma - \frac{a}{a-1}}|x|^{\frac{a}{a-1}}}
\end{equation*}
for some small constant $c_0 > 0$.  If $\gamma = \frac{a}{a-1}$, then noting that $\alpha > \frac{1}{2}$ and that $\alpha+\frac{1}{2}(a-2) > 0$, it follows that $|J_2| \lesssim e^{-\delta^a c_0 |x|^{\frac{a}{a-1}}}$ and the estimate is complete.  Otherwise, proceeding as before, for any $\beta > 0$,
\begin{equation*}
|J_2| \lesssim \frac{t^{\frac{1}{a-1}(\alpha - \frac{1}{2})}}{t^{\beta(\gamma - \frac{a}{a-1})}} \frac{1}{|x|^{\frac{1}{a-1}(\alpha + \frac{1}{2}(a-2) + \beta a)}}.
\end{equation*}
If $\gamma < \frac{a}{a-1}$, rewrite this as
\begin{equation*}
|J_2| \lesssim \frac{t^{\beta(\frac{a}{a-1} - \gamma)}}{t^{\frac{1}{a-1}(\frac{1}{2} - \alpha)}} \frac{1}{|x|^{\frac{1}{a-1}(\alpha + \frac{1}{2}(a-2) + \beta a)}}
\end{equation*}
and note that $\beta$ can be set as large as is desired to conclude a suitable estimate.

If $\gamma > \frac{a}{a-1}$, choose $\beta = \frac{\alpha - \frac{1}{2}}{(a-1)\gamma - a}$ as in the case of $|x| \leq C_0$, noting that it is still the case that this choice of $\beta$ is positive, as $\alpha > \frac{1}{2}$ and $\gamma > \frac{a}{a-1}$.  As before, it can thus be concluded that $|J_2| \lesssim \frac{1}{|x|^k}$ where $k = \frac{1}{a-1} \(\alpha \(\frac{(a-1)\gamma}{(a-1)\gamma - a}\) + \frac{1}{2} (a-2) - \frac{\frac{1}{2}a}{(a-1)\gamma - a} \)$ and it remains to show that $k > 1$ in this case.  However, since $\gamma > \frac{a}{a-1}$, it is necessarily the case that $\frac{(a-1)\gamma}{(a-1)\gamma - a} > 0$, hence the fact that $\alpha > \frac{1}{2} a \(1 - \frac{1}{\gamma}\)$ implies that
\begin{equation*}
k > \frac{1}{a-1} \(\frac{1}{2} a \(1 - \frac{1}{\gamma}\) \(\frac{(a-1)\gamma}{(a-1)\gamma - a}\) + \frac{1}{2}(a-2) - \frac{\frac{1}{2} a}{(a-1)\gamma - a} \) = 1,
\end{equation*}
which completes the estimate on $J_2$ and the proof of Lemma \ref{mainlemma}.

\section{Proof of Failure of Boundedness of the Maximal Operator for Regularity Below the Critical Index} \label{negsec}
To complete the proof of Theorem \ref{mainthm}, it remains to show that for $\gamma > 1$, the estimate $\Vert P^{*}_{a, \gamma} f \Vert_{L^2(\mathbb{R})} \lesssim \Vert f \Vert_{H^s(\mathbb{R})}$ cannot hold for $s < \frac{1}{4} a(1 - \frac{1}{\gamma})$.  In \cite{sjolin}, Sj\"olin proved this for $a = \gamma = 2$, generalising the aforementioned counterexample of Dahlberg and Kenig from \cite{dahlbergkenig}.  The proof given here is a further generalisation of this counterexample.

Fix $\gamma > 1$ and $s < \frac{1}{4} a (1 - \frac{1}{\gamma})$.  For each $v \in (0, v_0)$ for some small $v_0 > 0$, choose $g_v \in \mathcal{S}(\mathbb{R})$ to be a positive, even, real-valued function, supported in $[-v^{(a-1)-\frac{a}{\gamma}}, v^{(a-1) - \frac{a}{\gamma}}]$ and equal to $1$ on $[-\frac{1}{2} v^{(a-1) - \frac{a}{\gamma}}, \frac{1}{2} v^{(a-1) - \frac{a}{\gamma}}]$.  Define the function $f_v$ such that $\widehat{f_v}(\xi) = v g_v (v \xi + \frac{1}{v})$ and note that
\begin{eqnarray*}
\Vert f_v \Vert_{\dot{H}^s(\mathbb{R})}^2 &=& v^2 \int_{\mathbb{R}} \Big|g_v\Big(v \xi + \frac{1}{v}\Big)\Big|^2 |\xi|^{2s} \, d\xi\\*
&=& \frac{v^2}{v^{1+2s}} \int_{\mathbb{R}} \Big| g_v\Big(\xi + \frac{1}{v}\Big)\Big|^2 |\xi|^{2s} \, d\xi.
\end{eqnarray*}

For the integrand above to be non-zero, given the support of $g_v$, it is necessarily the case that $\xi + \frac{1}{v} \in [-v^{(a-1)-\frac{a}{\gamma}}, v^{(a-1) - \frac{a}{\gamma}}]$, hence $|\xi| \leq v^{(a-1) - \frac{a}{\gamma}} + \frac{1}{v}$.  Since $v \in (0, 1)$ and $(a-1) - \frac{a}{\gamma} > - 1$, given that $\gamma > 1$, it follows that $|\xi| \lesssim \frac{1}{v}$, so
\begin{equation*}
\Vert f_v \Vert_{\dot{H}^s(\mathbb{R})}^2 \lesssim v^{1-2s} v^{(a-1) - \frac{a}{\gamma}} v^{-2s} = v^{a - 4s - \frac{a}{\gamma}}.
\end{equation*}

Since $s < \frac{1}{4}a(1 - \frac{1}{\gamma})$ and $\Vert f_v \Vert_{H^s(\mathbb{R})} \sim \Vert f_v \Vert_{\dot{H}^s(\mathbb{R})}$, it can  be concluded that $\Vert f_v \Vert_{H^s(\mathbb{R})} \rightarrow 0$ as $v \rightarrow 0$.  It thus now suffices to show that there exists a choice of $t$, depending on $x$ and $v$, such that the $L^2(\mathbb{R})$ norm in $x$ of $P^t_{a, \gamma} f_v(x)$ is bounded below, uniformly in $v$.

Note first that
\begin{equation*}
P^t_{a, \gamma} f_v(x) = \int_{\mathbb{R}} e^{i(x \xi + t |\xi|^a)} e^{-t^{\gamma} |\xi|^a} v g_v\(v \xi + \frac{1}{v}\) \, d\xi.
\end{equation*}

Substituting $\eta = v\xi + \frac{1}{v}$ and removing a unimodular term that does not depend on $\eta$ from the integrand,
\begin{equation*}
|P^t_{a, \gamma} f_v(x)| = \Big|\int_{\mathbb{R}} e^{i(x \frac{\eta}{v} + t|\frac{\eta}{v} - \frac{1}{v^2}|^a)} e^{-t^{\gamma}|\frac{\eta}{v} - \frac{1}{v^2}|^a} g_v(\eta) \, d\eta \Big|.
\end{equation*}

Define
\begin{eqnarray*}
F_{x, t, v}(\eta) &\coloneq& x \frac{\eta}{v} + t\Big|\frac{\eta}{v} - \frac{1}{v^2}\Big|^a - \frac{t}{v^{2a}},\\
G_{t, v}(\eta) &\coloneq& t^{\gamma} \Big|\frac{\eta}{v} - \frac{1}{v^2}\Big|^a.
\end{eqnarray*}
Then it is clear, given the support of $g_v$, that
\begin{eqnarray*}
|P^t_{a, \gamma} f_v(x)| &=& \Big|\int_{-v^{(a-1) - \frac{a}{\gamma}}}^{v^{(a-1) - \frac{a}{\gamma}}} e^{iF_{x, t, v}(\eta)} e^{-G_{t, v}(\eta)} g_v(\eta) \, d\eta \Big|\\*
&\geq& \Big|\int_{-v^{(a-1) - \frac{a}{\gamma}}}^{v^{(a-1) - \frac{a}{\gamma}}} \cos(F_{x, t, v}(\eta)) e^{-G_{t, v}(\eta)} g_v(\eta) \, d\eta \Big|.
\end{eqnarray*}

By binomial expansion, for $|\eta| \leq v^{(a-1) - \frac{a}{\gamma}}$,
\begin{eqnarray*}
\Big|\frac{\eta}{v} - \frac{1}{v^2}\Big|^a &=& \(\frac{1}{v^2} - \frac{\eta}{v}\)^a\\*
&=& \frac{1}{v^{2a}} - \frac{a\eta}{v^{2(a-1) + 1}} + O\(\frac{\eta^2}{v^{2(a-2) + 2}}\),
\end{eqnarray*}
since $(a-1) - \frac{a}{\gamma} > -1$.  It follows that
\begin{equation*}
F_{x, t, v}(\eta) = x \frac{\eta}{v} - t a \frac{\eta}{v^{2a - 1}} + O\(\frac{t\eta^2}{v^{2(a-1)}}\).
\end{equation*}
Choose $x \in [0, v^{\frac{2a}{\gamma} - 2(a-1)}]$ and fix $t = \frac{xv^{2(a - 1)}}{a}$ (which is contained in $(0, 1)$, given the restriction on $x$).  Then
\begin{equation*}
F_{x, t, v}(\eta) = O (x \eta^2)
\end{equation*}
and hence
\begin{equation*}
F_{x, t, v}(\eta) \lesssim v^{\frac{2a}{\gamma} - 2(a-1) + 2((a-1) - \frac{a}{\gamma})} = 1.
\end{equation*}
For sufficiently small $v_0$, the implicit constant here may be set to $1$.

Additionally,
\begin{eqnarray*}
G_{t, v}(\eta) &=& x^{\gamma} v^{2\gamma(a - 1)} a^{-\gamma} O\(\frac{1}{v^{2a}}\)\\*
&=& O(x^{\gamma} v^{2a\gamma - 2\gamma - 2a}),
\end{eqnarray*}
hence
\begin{equation*}
G_{t, v}(\eta) \lesssim v^{2a - 2\gamma(a-1) + 2a\gamma - 2\gamma - 2a} = 1.
\end{equation*}

Given these estimates, it is clear that $\cos(F_{x, t, v}(\eta))$ and $e^{-G_{t, v}(\eta)}$ can be bounded below by constants for $|\eta| \leq v^{(a-1) - \frac{a}{\gamma}}$ and hence $|P^t_{a, \gamma} f_v(x)| \gtrsim v^{(a-1) - \frac{a}{\gamma}}$ for $x \in [0, v^{\frac{2a}{\gamma} - 2(a-1)}]$, so
\begin{equation*}
\Vert P^t_{a, \gamma} f_v \Vert_{L^2(\mathbb{R})}^2 \gtrsim v^{\frac{2a}{\gamma} - 2(a-1)} (v^{(a-1) - \frac{a}{\gamma}})^2 = 1
\end{equation*}
which completes the proof of the negative result.

\section{Further Remarks and the Proof of Theorem 1.3} \label{furtherremarks}
It is observed that in terms of determining the sharp exponent $s$ for which the estimate $\Vert P^{*}_{a, \gamma} f \Vert_{L^2(\mathbb{R})} \lesssim \Vert f \Vert_{H^s(\mathbb{R})}$ holds, the question of whether boundedness holds at the critical exponent, $s = \frac{1}{4} a (1 - \frac{1}{\gamma})$, for $a$, $\gamma > 1$, remains open.\footnote{For $a > 1$, $\gamma \in (0, 1]$, boundedness of $P^{*}_{a, \gamma}$ from $L^2(\mathbb{R})$ into $L^2(\mathbb{R})$ can be shown to hold by using similar methods to those given in the appendix to reduce the problem to boundedness of the Hardy--Littlewood maximal function in the case of $\gamma = 1$ and then applying Lemma \ref{boundinglemma} to conclude the same result for $\gamma \in (0, 1)$.}  In the case of the maximal operators with real-valued time, $S^{*}_a$, the problem of boundedness at the critical exponent, $s = \frac{a}{4}$ is also still open, even in the case of $a = 2$.  Nonetheless, it is remarked that the proof of boundedness of $P^{*}_{a, \gamma}$ given in Section \ref{possec} adapts without difficulty in the case of $\gamma = \frac{a}{a-1}$ to $s = \frac{1}{4} a (1 - \frac{1}{\gamma}) = \frac{1}{4}$.

A natural extension of Theorem \ref{mainthm} is to consider the values of $s$ for which a local norm bound on the maximal operator holds, that is to say
\begin{equation*}
\Vert P^{*}_{a, \gamma} f \Vert_{L^2([-1, 1])} \lesssim \Vert f \Vert_{H^s(\mathbb{R})}.
\end{equation*}
Denoting by $s_a^{\loc}(\gamma)$ the infimum of the values of $s > 0$ for which this estimate holds, the following analogue of Theorem \ref{mainthm} can be established:
\begin{thm} \label{mainthmloc}
For $\gamma \in (0, \infty)$ and $a > 1$, $s_a^{\loc}(\gamma) = \min\( \frac{1}{4} a \(1 - \frac{1}{\gamma}\)^{+}, \frac{1}{4} \)$.
\end{thm}
\begin{proof}
Observe that the global bounds from Theorem \ref{mainthm} automatically imply local bounds, so it is necessarily the case that $s_a^{\loc}(\gamma) \leq s_a(\gamma) = \frac{1}{4} a (1 - \frac{1}{\gamma})^{+}$.  Additionally, note that the counterexample given in Section \ref{negsec} is also a counterexample for the local estimate whenever the choices of $x$ are contained within $[-1, 1]$.  Since $x$ is chosen to be in $[0, v^{\frac{2a}{\gamma} - 2(a-1)}]$ for some small parameter $v$, this happens precisely when $\frac{2a}{\gamma} - 2(a-1) \geq 0$, that is when $\gamma \leq \frac{a}{a-1}$.  It follows that $s_a^{\loc}(\gamma) = \frac{1}{4} a (1 - \frac{1}{\gamma})^{+}$ for $\gamma \in (0, \frac{a}{a-1}]$.

In the proof of Lemma \ref{mainlemma}, if $\gamma \geq \frac{a}{a-1}$ and $x$ is small, the only requirement on $\alpha$ is that it is greater than or equal to $\frac{1}{2}$ (used in the estimate on $J_2$).  Consequently, for such $\gamma$, it must be the case that $s_a^{\loc}(\gamma) \leq \frac{1}{4}$.  Since the function $G_{t, v}(\eta)$ from Section \ref{negsec} is non-increasing in $\gamma$ (owing to the locality of $t$), the counterexample for $\gamma = \frac{a}{a-1}$, which shows that $s_a^{\loc}(\frac{a}{a-1}) \geq \frac{1}{4}$, also provides that $s_a^{\loc}(\gamma) \geq \frac{1}{4}$ for all $\gamma \geq \frac{a}{a-1}$, and so the theorem is established.
\end{proof}
It is remarked that the upper bound of $\frac{1}{4}$ for $s_a^{\loc}(\gamma)$ is perhaps unsurprising here in light of the fact that in 1987, Sj\"olin proved in \cite{sjolinarticle} that for all $a > 1$,
\begin{equation*}
\Vert S^{*}_a f \Vert_{L^2([-1, 1])} \lesssim \Vert f \Vert_{H^s(\mathbb{R})}
\end{equation*}
if and only if $s \geq \frac{1}{4}$.

In the case of this local problem, it can further be seen that $\Vert P^{*}_{a, \gamma} f \Vert_{L^2([-1, 1])} \lesssim \Vert f \Vert_{H^{\frac{1}{4}}(\mathbb{R})}$ holds for any $\gamma \geq \frac{a}{a-1}$ and $\Vert P^{*}_{a, \gamma} f \Vert_{L^2([-1, 1])} \lesssim \Vert f \Vert_{L^2(\mathbb{R})}$ holds for any $\gamma \in (0, 1]$ (that is boundedness holds at the critical index in these cases).  Given Theorem \ref{mainthmloc} together with this remark, the positive statements of Theorem \ref{convcorl2} follow from standard arguments deducing pointwise convergence from boundedness of maximal functions.  The negative statements are a consequence of the Nikishin--Stein maximal principle\footnote{This principle establishes that for appropriate operators, pointwise convergence results are in fact equivalent to weak bounds on maximal operators.  See also \cite{stein} and \cite{guzman}.}, as given in \cite{nikishin} and as applied by Dahlberg and Kenig in \cite{dahlbergkenig}, and the fact that the counterexample in Section \ref{negsec} can also be used to show failure of boundedness from $H^s(\mathbb{R})$ into $L^{2, \infty}([-1, 1])$.

\section*{Appendix: Proof of Lemma 1.4} \label{lemmaproof}
Observe that
\begin{eqnarray*}
S_a^{t + ih(t)} f(x) &=& \int_{\mathbb{R}} \fhat(\xi) e^{it|\xi|^a} e^{-h(t)|\xi|^a} e^{ix\xi} \, d\xi\\*
&=& \int_{\mathbb{R}} \fhat(\xi) e^{it|\xi|^a} e^{-g(t)|\xi|^a} e^{-(h(t) - g(t))|\xi|^a} e^{ix\xi} \, d\xi.
\end{eqnarray*}
Define $K: \mathbb{R} \rightarrow \mathbb{R}$ such that $\widehat{K}(\xi) = e^{-|\xi|^a}$ and for each $t \in [0, 1]$, define $K_t \coloneq t^{-1} K(t^{-1} \cdot)$, so that $\widehat{K_t} = \widehat{K}(t \cdot)$.  Then
\begin{eqnarray*}
S_a^{t + ih(t)} f(x) &=& \int_{\mathbb{R}} \fhat(\xi) e^{it|\xi|^a} e^{-g(t)|\xi|^a} \Big( \int_{\mathbb{R}} K_{(h(t) - g(t))^{\frac{1}{a}}}(y) e^{-iy\xi} \, dy \Big) e^{ix\xi} \, d\xi\\*
&=& \int_{\mathbb{R}} \Big( \int_{\mathbb{R}} \fhat(\xi) e^{it|\xi|^a} e^{-g(t)|\xi|^a} e^{i(x-y)\xi} \, d\xi \Big) K_{(h(t) - g(t))^{\frac{1}{a}}}(y) \, dy\\
&=& (S_a^{t + ig(t)}f) * K_{(h(t) - g(t))^{\frac{1}{a}}}(x),
\end{eqnarray*}
so
\begin{equation*}
\sup_{t \in (0, 1)} |S_a^{t + ih(t)} f(x)| = \sup_{t \in (0, 1)} |(S_a^{t + ig(t)}f) * K_{(h(t) - g(t))^{\frac{1}{a}}} (x)| \leq \sup_{u \in (0, 1)} | (\sup_{t \in (0, 1)} |S_a^{t + ig(t)} f|) * K_u(x) |.
\end{equation*}
Now,
\begin{equation*}
K(x) = \frac{1}{2\pi} \int_{\mathbb{R}} e^{-|\xi|^a} e^{ix\xi} \, d\xi = \frac{1}{\pi} \int_0^{\infty} e^{-\xi^a} \cos(x\xi) \, d\xi.
\end{equation*}
Since $\int_0^{\infty} e^{-\xi^a} \, d\xi = \Gamma(\frac{1}{a} + 1) < \infty$, it is clear that $K \in L^{\infty}(\mathbb{R})$.  Further, integrating by parts twice and using that for any $c > -1$ it is also true that $\int_0^{\infty} \xi^c e^{-\xi^a} \, d\xi = \frac{1}{a} \Gamma(\frac{c+1}{a}) < \infty$, it follows for any $x \neq 0$ that $|K(x)| \lesssim |x|^{-2}$.  It is thus the case that
\begin{equation*}
\sup_{u \in (0, 1)} | (\sup_{t \in (0, 1)} |S_a^{t + ig(t)} f|) * K_u(x) | \lesssim M(\sup_{t \in (0, 1)} |S_a^{t + ig(t)} f|)(x),
\end{equation*}
where $M$ is the Hardy--Littlewood maximal function.  By boundedness of $M$ on $L^2(\mathbb{R})$, the lemma follows.

\renewcommand{\bibname}{Bibliography} 

\bibliographystyle{adb}
\bibliography{complschr}

\end{document}